\documentclass{gtmon_a}
\pdfoutput=1

\proceedingstitle{Proceedings of the Nishida Fest (Kinosaki 2003)}
\conferencestart{28 July 2003}
\conferenceend{8 August 2003}
\conferencename{International Conference in Homotopy Theory}
\conferencelocation{Kinosaki, Japan}

\editor{Matthew Ando}
\givenname{Matthew}
\surname{Ando}

\editor{Norihiko Minami}
\givenname{Norihiko}
\surname{Minami}

\editor{Jack Morava}
\givenname{Jack}
\surname{Morava}

\editor{W Stephen Wilson}
\givenname{W Stephen}
\surname{Wilson}

\title{Spaces of real polynomials with common roots}

\author{Yasuhiko Kamiyama}
\givenname{Yasuhiko}
\surname{Kamiyama}
\address{Department of Mathematics\\
University of the Ryukyus\\\newline
Nishihara-Cho\\
Okinawa 903-0213\\
Japan}
\email{kamiyama@sci.u-ryukyu.ac.jp}
\urladdr{}

\volumenumber{10}
\issuenumber{}
\publicationyear{2007}
\papernumber{13}
\startpage{227}
\endpage{235}

\doi{}
\MR{}
\Zbl{}

\arxivreference{}

\keyword{real polynomial}
\keyword{loop space}
\keyword{stable splitting}
\subject{primary}{msc2000}{55P15}
\subject{secondary}{msc2000}{55P35}
\subject{secondary}{msc2000}{58D15}

\received{13 August 2004}
\revised{9 May 2005}
\accepted{}
\published{29 January 2007}
\publishedonline{29 January 2007}
\proposed{}
\seconded{}
\corresponding{}
\version{}

\makeatletter
\def\cnewtheorem#1[#2]#3{\newtheorem{#1}{#3}[section]
\expandafter\let\csname c@#1\endcsname\c@thm}

\let\xysavmatrix\xymatrix
\def\xymatrix{\disablesubscriptcorrection\xysavmatrix}
\AtBeginDocument{\let\bar\wbar\let\tilde\wtilde\let\hat\what}

\theoremstyle{plain}
\newtheorem{thm}{Theorem}[section]
\cnewtheorem{lem}[thm]{Lemma}
\cnewtheorem{cor}[thm]{Corollary}
\newtheorem{prop}[thm]{Proposition}
\newtheorem{alphthm}{Theorem}

\newtheorem{alphprop}{Proposition}

\let\c@alphprop\c@alphthm
\makeautorefname{alphthm}{Theorem}
\makeautorefname{alphprop}{Proposition}
\theoremstyle{definition}
\newtheorem{exmp}{Example}[section]
\numberwithin{equation}{section}

\makeatother  

\newcommand{\bv}[2]{\bigvee_{#1}^{#2}}
\newcommand{\bvs}[2]{\underset{#2}{\underset{#1}{\bigvee}}} 
\renewcommand{\C}{\mathbb{C}}
\newcommand{\CP}[1]{\C P^{#1}}
\renewcommand{\D}[2]{D_{#1} (S^{#2})}
\newcommand{\Dx}[3]{D_{#1} \xi^{#2} (S^{#3})}
\newcommand{\Ga}[2]{\left[ \frac{#1}{#2} \right]}
\newcommand{\Map}[2]{\text{\rm Map}_{#1}^T (\CP{1}, \CP{#2})}
\renewcommand{\mod}{\, \text{\rm mod} \,}
\newcommand{\ol}[1]{\overline{#1}}
\renewcommand{\R}[2]{\text{\rm Rat}_{#1} (\CP{#2})}
\renewcommand{\Re}{\mathbb{R}}
\newcommand{\RP}[1]{\Re P^{#1}}
\newcommand{\RX}[3]{RX_{#1, #3}^{#2}}
\newcommand{\simeqs}{\underset{s}{\simeq}}
\newcommand{\SP}[1]{\text{\rm SP}^{#1}}
\renewcommand{\TH}{\widetilde{H}}
\newcommand{\X}[3]{X_{#1, #3}^{#2}}
\newcommand{\Y}[3]{Y_{#1, #3}^{#2}}
\newcommand{\Zp}{\mathbb{Z}/p}
\newcommand{\mapright}[1]{\smash{\mathop{
\hbox to 1cm{\rightarrowfill}}\limits^{#1}}}

\renewcommand{\theenumi}{\roman{enumi}}
\begin{document}

\begin{htmlabstract} 
Let RX<sub>k,n</sub><sup>l</sup> be the space consisting of all
(n+1)&ndash;tuples (p<sub>0</sub>
(z),&hellip;,p<sub>n</sub>(z)) of monic polynomials over <b>R</b> of degree
k and such that there are at most l roots common to all p<sub>i</sub>(z).
In this paper, we prove a stable splitting of RX<sub>k,n</sub><sup>l</sup>.
\end{htmlabstract}

\begin{abstract} 
Let $RX_{k,n}^{l}$ be the space consisting of all $(n+1)$--tuples $(p_0
(z), \ldots, p_n (z))$ of monic polynomials over $\mathbb{R}$ of degree
$k$ and such that there are at most $l$ roots common to all $p_i (z)$.
In this paper, we prove a stable splitting of $RX_{k,n}^{l}$.
\end{abstract}

\maketitle

\section{Introduction} 
\label{sec:sec1}

Let $\R{k}{n}$ denote the space of based holomorphic maps of degree $k$ 
from the Riemannian sphere $S^2= \C \cup \infty$ to the complex projective 
space $\CP{n}$. 
The basepoint condition we assume is that $f(\infty) = [1, \dots, 1]$. 
Such holomorphic maps are given by rational functions:
\begin{multline*}
\R{k}{n}  =  \bigl\{ (p_0 (z), \dots, p_n (z)): 
\text{each $p_i(z)$ is a monic polynomial over $\C$}\\
\text{of degree $k$ and such that there are no roots common to all 
$p_i (z)$} \bigr\}.
\end{multline*}
There is an inclusion $\R{k}{n} \hookrightarrow \Omega_k^2 \CP{n} \simeq 
\Omega^2 S^{2n+1}$. Segal \cite{S} proved that the inclusion 
is a homotopy equivalence up to dimension $k (2n-1)$. 
Later, the stable homotopy type of $\R{k}{n}$ was described by
Cohen et al \cite{CCMM1,CCMM2} as follows. Let $\Omega^2 S^{2n+1} \simeqs 
\bv{1 \leq q}{}
\D{q}{2n-1}$ be Snaith's stable splitting of 
$\Omega^2 S^{2n+1}$. 
Then 
\begin{equation}
\R{k}{n} \simeqs \bv{q=1}{k} \D{q}{2n-1}. 
\label{eqn:eqn1.1}
\end{equation}
In Kamiyama \cite{K}, \eqref{eqn:eqn1.1} was generalized as follows. We set 
\begin{multline*}
\X{k}{l}{n} =  \bigl\{ (p_0 (z), \dots, p_n (z)): 
\text{each $p_i(z)$ is a monic polynomial over $\C$}\\
\text{of degree $k$ and such that there are at most $l$ 
roots common to all $p_i (z)$} \bigr\}.
\end{multline*}
In particular, $\X{k}{0}{n} = \R{k}{n}$. 
Let 
$$J^l (S^{2n}) \simeq S^{2n} \cup e^{4n} \cup e^{6n} \cup \ldots \cup 
e^{2ln} \subset \Omega S^{2n+1}$$ 
be the $l$th stage of 
the James filtration of $\Omega S^{2n+1}$, and let $W^l (S^{2n})$ be the 
homotopy theoretic fiber of the inclusion $J^l (S^{2n}) \hookrightarrow 
\Omega S^{2n+1}$. We generalize Snaith's stable splitting of 
$\Omega^2 S^{2n+1}$ as follows: 
$$W^l (S^{2n}) \simeqs \bv{1 \leq q}{} \Dx{q}{l}{2n}.$$
Then we have a stable splitting 
\begin{equation*}
\X{k}{l}{n} \simeqs \bv{q=1}{k} \Dx{q}{l}{2n}. 
\end{equation*}
The purpose of this paper is to study the real part $\RX{k}{l}{n}$ of 
$\X{k}{l}{n}$ and prove a stable splitting of this. More precisely, 
let $\RX{k}{l}{n}$ be the subspace of $\X{k}{l}{n}$
 consisting of elements 
$(p_0(z), \dots, p_n(z))$ such that each $p_i(z)$ has real coefficients.
Our main results 
will be stated in \fullref{sec:sec2}. Here we give a theorem which generalizes 
\eqref{eqn:eqn1.1}. 
Since the homotopy type of $\RX{k}{0}{1}$ is known 
(see \fullref{eg:eg2.1} (iii)), 
we assume $n \geq 2$. In this case, there is an inclusion 
$$\RX{k}{0}{n} \hookrightarrow \Omega S^n \times \Omega^2 S^{2n+1}.$$ 
(See \fullref{lem:lem3.1}.)
\begin{thm} \label{thm:thm1.1} For $n \geq 2$, we define the weight of 
stable summands in $\Omega S^n$ as usual, but those in $\Omega^2 S^{2n+1}$
we define as being twice the usual one. 
Then $\RX{k}{0}{n}$ is stably homotopy equivalent 
to the collection of stable summands in 
$\Omega S^n \times \Omega^2 S^{2n+1}$ of weight $\leq k$. Hence, 
$$\RX{k}{0}{n} \simeqs \bv{p+2q \leq k}{} \Sigma^{p (n-1)} \D{q}{2n-1} 
\vee \bv{p=1}{k} S^{p (n-1)}.$$
\end{thm}
This paper is organized as follows. In \fullref{sec:sec2} 
we state the main results. 
We give a stable splitting of $\smash{\RX{k}{l}{n}}$ in \fullref{thm:thmA}
and \fullref{thm:thmB}. 
In order to prove these theorems, we also consider a space 
$\smash{\Y{k}{l}{n}}$,
which is an open set of $\smash{\RX{k}{l}{n}}$. 
We give a stable splitting of $\smash{\Y{k}{l}{n}}$ 
in \fullref{prop:propC}. In \fullref{sec:sec3} we prove \fullref{prop:propC}. 
In \fullref{sec:sec4} we prove \fullref{thm:thmA}
and \fullref{thm:thmB}. 

\section{Main results} 
\label{sec:sec2}

We set 
$$\Y{k}{l}{n} = \bigl\{ (p_0 (z), \dots, p_n (z)) \in \RX{k}{l}{n} : 
\text{there are no {\it real} roots common to all } p_i(z) \bigr\}.$$
The spaces $\Y{k}{l}{n}$ and $\RX{k}{l}{n}$ are in the following relation: 
\begin{equation*}
\setcounter{MaxMatrixCols}{13}
\begin{matrix} 
\Y{k}{k}{n} & \supset & \Y{k}{k-1}{n} & \supset & \cdots & \supset & 
\Y{k}{l}{n}
& \supset & \cdots & \supset & \Y{k}{1}{n} & = & \Y{k}{0}{n}\\
\cap & {}  & \cap & {} & {} & {} & \cap  & {} & {} & {} & \cap 
& {} & \parallel \\
\RX{k}{k}{n} & \supset & \RX{k}{k-1}{n} & \supset & \cdots & \supset & 
\RX{k}{l}{n}
& \supset & \cdots & \supset & \RX{k}{1}{n} & \supset & \RX{k}{0}{n}
\end{matrix}
\end{equation*}
where each subset is an open set. Moreover, $\Y{k}{2i+1}{n}= 
\Y{k}{2i}{n}$. In fact, if $\alpha \in H_+$ (where $H_+$ is 
the open upper half-plane) is a root of a real polynomial, then so is 
$\ol{\alpha} \in H_-$. 

We have the following examples. 
\begin{exmp}\label{eg:eg2.1} 
$\qua$\newline
\begin{enumerate}
\item It is proved by Mostovoy \cite{M} that 
$\Y{k}{k}{1}$ consists of $k+1$ contractible connected components. 
\item  The following result is proved by Vassiliev \cite{V}.
For $n \geq 3$, there is a 
homotopy equivalence $\Y{k}{k}{n} \simeq J^k (S^{n-1})$,  
where $J^k (S^{n-1})$ is as above the $k$th stage of 
the James filtration of $\Omega S^n$. For $n =2$, these 
spaces are stably homotopy equivalent. 
\item It is proved by Segal \cite{S} that 
$$\RX{k}{0}{1} \simeq \coprod_{q = 0}^k  
\R{\min (q,k-q)}{1}.$$
\item $\RX{k}{k-1}{n} \cong \Re^k \times (\Re^{kn})^\ast$ and 
$\RX{k}{k}{n} \cong \Re^{k(n+1)}$. 
\end{enumerate}
In fact, $(p_0(z), \dots, 
p_n(z)) \in \RX{k}{k}{n}$ is an element of 
$\RX{k}{k-1}{n}$ if and only if $p_i(z) \not= p_j(z)$ for some $i, j$. 
Hence, the first homeomorphism holds. 
\end{exmp}
Now we state our main results. 
\begin{alphthm} \label{thm:thmA} For $n \geq 1$ and $i \geq 0$, there is a  
homotopy equivalence
$$\RX{k}{2i+1}{n} \simeq \X{\Ga{k}{2}}{i}{n},$$
where $\bigl[\frac{k}{2}\bigr]$ denotes as usual the largest integer $\leq \frac{k}{2}$. 
\end{alphthm}
\begin{alphthm} \label{thm:thmB}
For $n \geq 1$ and $i \geq 0$, there is a 
stable homotopy equivalence
$$\RX{k}{2i}{n} \simeqs \X{\Ga{k}{2}}{i}{n} \vee \Sigma^{2in} 
\Biggl(\bvs{p+2q \leq k-2i}{1 \leq p} \Sigma^{p(n-1)} 
\D{q}{2n-1} \vee \bv{p=1}{k-2i} S^{p (n-1)} \Biggr).$$
\end{alphthm}
We study $\smash{\RX{k}{l}{n}}$ by induction with making $l$ larger. Hence, 
the induction starts from $\RX{k}{0}{n}$. 
Recall that $\RX{k}{0}{n}= \Y{k}{0}{n}$. We study $\Y{k}{l}{n}$ 
by induction with making $l$ smaller, 
where the initial condition is given in \fullref{eg:eg2.1} (ii). 
In fact, we have the following proposition.
\begin{alphprop}\label{prop:propC}\

\begin{enumerate}
\item For $n \geq 2$, we define the weight of 
stable summands in $\Omega S^n$ as usual, but those in $W^i (S^{2n})$
we define as being 
twice the usual one. Then $\Y{k}{2i}{n}$ is stably homotopy equivalent 
to the collection of stable summands in 
$\Omega S^n \times W^i (S^{2n})$ of weight $\leq k$. Hence,
$$\Y{k}{2i}{n} \simeqs \bv{p+2q \leq k}{} \Sigma^{p(n-1)} \Dx{q}{i}{2n}
\vee \bv{p=1}{k} S^{p (n-1)}.$$
\item When $n=1$, there is a homotopy equivalence 
$$\Y{k}{2i}{1} \simeq \coprod_{q = 0}^k \X{\min (q,k-q)}{i}{1}.$$
\end{enumerate}
\end{alphprop}
Note that \fullref{prop:propC} (ii) is a generalization of 
\fullref{eg:eg2.1} (i) and (iii). 

\section[Proof of proposition]{Proof of \fullref{prop:propC}}
\label{sec:sec3}

We study the space of continuous maps which contains  
$\Y{k}{k}{n}$ or $\RX{k}{0}{n}$. For simplicity, we assume 
$n \geq 2$. (The case for $n=1$ can be obtained by slight modifications.) 
Each $f \in \Y{k}{k}{n}$ defines a map $f \co S^1 \rightarrow \RP{n}$, 
where $S^1 = \Re \cup \infty$. Hence, there is a natural map 
$$\Y{k}{k}{n} \rightarrow \Omega_{k \mod 2} \RP{n} \simeq 
\Omega S^n.$$ 
\fullref{eg:eg2.1} (ii) implies that $\Y{k}{k}{n}$ is the $k(n-1)$--skeleton 
of $\Omega S^n$. 

On the other hand, let $\smash{\Map{k}{n}}$ be the space of 
continuous basepoint-preserving 
conjugation-equivariant maps of degree $k$ from $\CP{1}$ to $\CP{n}$. 
Then there is an inclusion 
$$\RX{k}{0}{n} \hookrightarrow \Map{k}{n}.$$
\begin{lem} \label{lem:lem3.1} 
For $n \geq 2$, $\Map{k}{n} \simeq \Omega S^n \times 
\Omega^2 S^{2n+1}$. 
\end{lem}
\begin{proof}
It is easy to see that 
$$\Map{k}{n} \simeq \Map{0}{n}.$$ 
Since $\Map{0}{n}$ can be thought as the space of maps 
$$(D^2, S^1, \ast) \rightarrow (\CP{n}, \RP{n}, \ast)$$ of degree $0$, 
there is a fibration 
$$\Omega^2 S^{2n+1} \rightarrow \Map{0}{n} \rightarrow \Omega S^n.$$
This is a pullback of the path fibration $\Omega^2 S^{2n+1} 
\rightarrow P \Omega S^{2n+1} \rightarrow \Omega S^{2n+1}$ by the map 
$\Omega f\co \Omega S^n \rightarrow \Omega S^{2n+1}$, where 
$f\co S^n \rightarrow S^{2n+1}$ is a lift of the inclusion 
$\RP{n} \hookrightarrow \CP{n}$. Since $f$ is null homotopic, 
the fibration is trivial. This completes the proof of 
\fullref{lem:lem3.1}. 
\end{proof}

Hereafter, every homology is with $\Zp$--coefficients, 
where $p$ is a prime. 
Recall that for $n \geq 2$, we have $H_\ast (\Omega S^n) 
\cong \Zp \left[ x_{n-1} \right]$. We define the weight of $x_{n-1}$ by 
$w (x_{n-1})=1$. 
On the other hand, we define the weight of 
an element of $H_{\ast}\bigl(\smash{\X{k}{i}{n}}\bigr)$ as being twice the usual one. 
For example, let $y_{2(l+1)n-1}$ be the generator of 
$\TH_\ast\bigl(\smash{\X{k}{l}{n}}\bigr)$ 
of least degree. The usual weight of $y_{2(l+1)n-1}$ is $l+1$, 
but we reset 
$w (y_{2(l+1)n-1}) = 2(l+1)$. 
\begin{prop}\label{prop:prop3.2}
For $n \geq 2$, 
$H_\ast\bigl(\smash{\Y{k}{2i}{n}}\bigr)$ is isomorphic to the subspace of 
$H_{\ast}\bigl(\Omega S^n{\times}\smash{\X{k}{i}{n}}\bigr)$ 
spanned by monomials of weight $\leq k$.
\end{prop}
We prove the proposition from the following lemma. 
\begin{lem} \label{lem:lem3.3}
We have the following long exact sequence: 
$$\cdots  \rightarrow  H_\ast (\Y{k}{2i-2}{n}) \rightarrow 
H_\ast (\Y{k}{2i}{n}) 
\overset{\phi}{\rightarrow} H_{\ast-2in}(\RX{k-2i}{0}{n}) 
\rightarrow 
H_{\ast-1}(\Y{k}{2i-2}{n}) \rightarrow \cdots.$$  
\end{lem}
\begin{proof}
In \cite[Propositions 4.5 and 5.4]{K}, we constructed a 
similar long exact sequence 
from the fact that 
\begin{equation*}
\X{k}{l}{n} -\X{k}{l-1}{n} = \C^l  \times \R{k-l}{n}, 
\end{equation*}
where $\C^l \times \R{k-l}{n}$ corresponds to the subspace of
$\X{k}{l}{n}$ consisting of elements $(p_0(z),$ $\dots, p_n(z))$ 
such that there are exactly $l$ roots common to all $p_i(z)$. 
The proposition is proved 
similarly using the fact that 
$$\Y{k}{2i}{n} - \Y{k}{2i-2}{n} \cong \SP{i} (H_+) \times \RX{k-2i}{0}{n},$$ 
where $\SP{i} (H_+)$ denotes the $i$th symmetric product of $H_+$.   
\end{proof}

\begin{proof}[Proof of \fullref{prop:prop3.2}]
In order to prove \fullref{prop:prop3.2} by induction, we introduce 
the following total order $\leq$ to $\Y{k}{2i}{n}$ for $k \geq 1$ 
and $i \geq 0$:
$\Y{k}{2i}{n} < \Y{k'}{2 i'}{n}$ if and only if 
\begin{enumerate}
\item $k < k'$, or
\item $k= k'$ and $i > i'$. 
\end{enumerate}
By \fullref{eg:eg2.1} (ii), \fullref{prop:prop3.2} holds for $\Y{k}{k}{n}$. 
Assuming that \fullref{prop:prop3.2} holds for $\Y{k}{2i}{n}$ 
and $\RX{k-2i}{0}{n}$, we prove for $\Y{k}{2i-2}{n}$. 
We have the following long exact 
sequence: 
\begin{multline} \label{eqn:eqn3.1}
\cdots \longrightarrow H_\ast \Bigl(\X{\Ga{k}{2}}{i-1}{n}\Bigr) \longrightarrow
H_\ast \Bigl(\X{\Ga{k}{2}}{i}{n}\Bigr) \\
\overset{\psi}{\longrightarrow} H_{\ast-2in}\Bigl(\R{\Ga{k}{2} -i}{n}\Bigr)
\overset{\theta}{\longrightarrow} 
H_{\ast-1}\Bigl(\X{\Ga{k}{2}}{i-1}{n}\Bigr) \longrightarrow \cdots.
\end{multline} 
For $n \geq 2$, we consider the homomorphism 
$$1 \otimes \psi \co H_\ast (\Omega S^n) \otimes H_\ast
\Bigl(\X{\Ga{k}{2}}{i}{n}\Bigr) \rightarrow H_\ast (\Omega S^n) \otimes
  H_{\ast-2in}\Bigl(\R{\Ga{k}{2}-i}{n}\Bigr).$$ 
Restricting the domain to $H_\ast (\Y{k}{2i}{n})$, we obtain the 
homomorphism $\phi$ in \fullref{lem:lem3.3}.
Now it is easy to prove \fullref{prop:prop3.2}. 
\end{proof}
\begin{proof}[Proof of \fullref{prop:propC} (i)] 
We construct a stable map from the 
right-hand side of \fullref{prop:propC} (i) to $\Y{k}{2i}{n}$. 
Since our constructions are similar, 
we construct a stable map 
$g_{p, q, i, n} \co \Sigma^{p (n-1)} \smash{\Dx{q}{i}{2n} \rightarrow 
\Y{k}{2i}{n}}$.
First, using the fact that 
$\smash{\RX{1}{0}{n}} \simeq S^{n-1}$ (see \fullref{eg:eg2.1} (iv)), 
there is a stable map $f_{p,n} \co S^{p(n-1)} \rightarrow
\smash{\RX{p}{0}{n}}$.
Second, there is a stable section 
$e_{q,i,n}\co \Dx{q}{i}{2n} \rightarrow \X{q}{i}{n}$. Third, there 
is an inclusion 
\begin{equation} \label{eqn:eqn3.2}
\eta_{q,i,n}\co \X{q}{i}{n} \hookrightarrow \Y{2q}{2i}{n}. 
\end{equation}
To construct this, 
we fix a homeomorphism $h : \C \overset{\cong}{\rightarrow} H_+$. 
For $(p_0(z), \dots, p_n(z))$
\newline
$\in \X{q}{i}{n}$, we write 
$p_j(z)= \prod_{s=1}^q (z- \alpha_{s,j})$. Then we set 
\begin{multline*}
\eta_{q,i,n} (p_0(z), \dots, p_n(z))\\
=\biggl( \prod_{s=1}^q (z- h(\alpha_{s,0})) (z- \ol{h (\alpha_{s,0})}), 
\dots, \prod_{s=1}^q (z- h(\alpha_{s,n})) (z- \ol{h (\alpha_{s,n})})
\biggr).
\end{multline*}
Now consider the following composite of maps 
\begin{equation} \label{eqn:eqn3.3}
S^{p(n-1)} \times \Dx{q}{i}{2n} 
\mapright{f_{p, n} \times 
(\eta_{q,i,n} \circ e_{q,i,n})}
\RX{p}{0}{n} \times \Y{2q}{2i}{n} 
\overset{\mu}{\rightarrow} \Y{p+2q}{2i}{n} 
\hookrightarrow \Y{k}{2i}{n}, 
\end{equation}
where $\mu$ is a loop sum which is constructed in the same way as in 
the loop sum $\R{k}{n} \times \R{l}{n} \rightarrow \R{k+l}{n}$ in 
Boyer--Mann \cite{BM}. 
We can construct $g_{p, q, i, n}$ from \eqref{eqn:eqn3.3}. 

Note that the stable map for \fullref{prop:propC} (i) is 
compatible with the homology splitting by weights. Using 
\fullref{prop:prop3.2},  
it is easy to show that this map induces an isomorphism in homology, 
hence is a stable homotopy equivalence. 
This completes the proof of \fullref{prop:propC} (i). 
\end{proof}
\begin{proof}[Proof of \fullref{prop:propC} (ii)]
By a similar argument to the proof of \fullref{prop:prop3.2}, we can calculate 
$H_\ast (\Y{k}{2i}{1})$. Then we can construct an unstable map from the 
right-hand side of \fullref{prop:propC} (ii) to $\Y{k}{2i}{1}$ 
in the same way as in \fullref{prop:propC} (i). 
\end{proof}

\section[Proofs of Theorems A and B]{Proof of \fullref{thm:thmA} and 
\fullref{thm:thmB}} 
\label{sec:sec4}

\begin{prop} \label{prop:prop4.1}
The homologies of the both sides of \fullref{thm:thmA} or
\fullref{thm:thmB} are isomorphic. 
\end{prop}
\begin{proof} We prove the proposition about $\RX{k}{l}{n}$ 
by induction with making $l$ larger. 
As in \fullref{lem:lem3.3}, there is a long exact sequence 
\begin{multline*}
\cdots  \longrightarrow  H_\ast \bigl(\RX{k}{l}{n}\bigr) \longrightarrow H_\ast 
\bigl(\RX{k}{l+1}{n}\bigr) \\
\longrightarrow  H_{\ast-(l+1)n} \bigl(\RX{k-(l+1)}{0}{n}\bigr)
\overset{\Theta}{\longrightarrow}
H_{\ast-1} \bigl(\RX{k}{l}{n}\bigr) \longrightarrow \cdots. 
\end{multline*}
This sequence is constructed from 
the following decomposition as sets 
$$\RX{k}{l+1}{n} -\RX{k}{l}{n} = \coprod_{a+2b= l+1}
\SP{a} (\Re) \times \SP{b} (H_+) \times \RX{k-(l+1)}{0}{n}$$ 
and the fact that $H_c^\ast (\SP{a} (\Re))=0$ for $a \geq 2$, 
where $H_c^\ast$ is the cohomology with compact 
supports. 

Assuming that the proposition holds for $l \leq 2i+1$, we determine 
$H_\ast \bigl(\RX{k}{2i+2}{n}\bigr)$. 
The homomorphism $\Theta$ is given as follows. 
Note that \fullref{thm:thmB} is equivalent to 
\begin{equation} \label{eqn:eqn4.1}
\RX{k}{2i}{n} \simeqs \X{\Ga{k}{2}}{i}{n} \vee 
\Sigma^{(2i+1)n-1} \bigl(\RX{k-2i-1}{0}{n} \vee S^0\bigr). 
\end{equation}
From inductive hypothesis, we have  
\begin{multline} \label{eqn:eqn4.2}
H_{\ast-(2i+2)n} \bigl(\RX{k-2i-2}{0}{n}\bigr) \cong \\
 H_{\ast-(2i+2)n} \Bigl( \R{\Ga{k}{2}-(i+1)}{n}\Bigr)
 \oplus \TH_{\ast-(2i+2)n} \bigl(\Sigma^{n-1} \RX{k-2i-3}{0}{n}
\vee S^{n-1}\bigr)
\end{multline}
and 
\begin{equation} \label{eqn:eqn4.3}
H_{\ast-1} (\RX{k}{2i+1}{n}) \cong 
H_{\ast-1} \Bigl(\X{\Ga{k}{2}}{i}{n}\Bigr). 
\end{equation}  
Recall the homomorphism $\theta$ in \eqref{eqn:eqn3.1}
with $i$ replaced by $i+1$. 
Then $\Theta \co \text{\eqref{eqn:eqn4.2}} \rightarrow \text{\eqref{eqn:eqn4.3}}$ 
is given by mapping 
the first summand by $\theta$ and the second summand by $0$. Hence, 
$H_\ast \bigl(\RX{k}{2i+2}{n}\bigr)$ is isomorphic to the homology of the 
right-hand side of \eqref{eqn:eqn4.1} with $i$ replaced by $i+1$. 

By a similar argument, we can determine $H_\ast \bigl(\RX{k}{2i+1}{n}\bigr)$ 
inductively by assuming the truth of the proposition for $l \leq 2i$. 
 This completes the proof of \fullref{prop:prop4.1}.  
\end{proof}

Finally, we construct an unstable map (resp. a stable map) from 
the right-hand side of \fullref{thm:thmA} (resp. \eqref{eqn:eqn4.1}) to 
$\RX{k}{2i+1}{n}$ (resp. $\RX{k}{2i}{n}$). 
First, the unstable map from the 
right-hand side of \fullref{thm:thmA} 
or the first stable summand in \eqref{eqn:eqn4.1} 
is essentially the inclusion
$$\X{q}{i}{n} \overset{\eta_{q, i, n}}{\longrightarrow}
\Y{2q}{2i}{n} \subset \RX{2q}{2i}{n},$$
where $\eta_{q, i, n}$ is defined in \eqref{eqn:eqn3.2}. 
Next, the stable map from the second 
stable summand in \eqref{eqn:eqn4.1} is constructed in the same way as in 
$g_{p, q, i, n}$ (see \eqref{eqn:eqn3.3}) using the fact that 
$\RX{2i+1}{2i}{n} \simeq S^{(2i+1)n-1}$ (see \fullref{eg:eg2.1} (iv)).  
This completes the proofs of \fullref{thm:thmA} and \fullref{thm:thmB}. \qed

\bibliographystyle{gtart}
\bibliography{link}

\end{document}